\documentclass[a4paper,12pt,reqno]{amsart}

\usepackage{color}%for comments

\usepackage{amsmath,amstext,amssymb,amsopn,amsthm}

\usepackage[margin=30mm]{geometry}
\usepackage{eucal,mathrsfs,dsfont}%gives nicer set-names. Use \mathds instead of \mathds
\usepackage{soul}
\usepackage{color}
\definecolor{mr}{rgb}{0.1,0.2,0.7}
\definecolor{tk}{rgb}{0.7,0.1,0.2}

\newtheorem{theorem}{Theorem}[section]

\newtheorem{lemma}[theorem]{Lemma}

\theoremstyle{definition}

\theoremstyle{remark}
\newtheorem{remark}[theorem]{Remark}
\newtheorem{example}[theorem]{Example}

\renewcommand{\mathcal}{\mathscr}
\newcommand{\E}{\mathds{E}}
\renewcommand{\P}{\mathds{P}}
\newcommand{\eps}{\varepsilon}

\newcommand{\Pp}{\mathds{P}} % Probability
\newcommand{\R}{\mathds{R}}

\newcommand{\wt}{\widetilde}

\newcommand{\di}{\mathrm{d}}
\newcommand{\be}{\begin{equation}}
\newcommand{\ee}{\end{equation}}
\newcommand{\ba}{\begin{aligned}}
\newcommand{\ea}{\end{aligned}}

\newcommand{\D}{\mathds{D}}

\DeclareMathOperator{\supp}{supp}

\title[Support theorem for L\'evy driven SDEs]{Support theorem for L\'evy driven SDEs}

\author[A. Kulik]{Alexei Kulik}

\address{Faculty of Pure and Applied Mathematics, Wroc{\l}aw University of Science and Technology, Wyb. Wyspia{\'n}skiego 27, 50-370 Wroc{\l}aw, Poland.}
\begin{document}

\begin{abstract} We provide a support theorem for the law of the solution to an SDE with jump noise. This theorem applies to  general SDEs with jumps and is illustrated by examples of SDEs with quite degenerate jump noises, where the theorem leads to an informative description of the support.
\end{abstract}

\maketitle

\section{Introduction}  The aim of this note is to describe the topological support of the law of the solution to a general SDE with jump noise. For diffusions such a description is provided by the classical Stroock-Varadhan support theorem, see \cite{StroockWaradhan}. The natural question of an extension of this theorem to SDEs with jumps was studied intensively at the end of 90-ies -- early 2000-ies by H.Kunita, Y.Ishikawa, and T.Simon, see \cite{Kunita}, \cite{Simon}, \cite{Ishikawa}. It appears that, in the jump noise setting, the two cases should be naturally separated, named in \cite{Simon}  the \emph{`Type I'} and \emph{`Type II'} SDEs. In plain words, an SDE is of `Type I' if the small jumps are absolutely integrable and of `Type II' otherwise. A `Type I' SDE admits a simple and intuitively clear description of the support; this case was studied completely in \cite{Simon}. `Type II' SDEs  had been studied under some additional limitations only, which can be understood as certain technical  \emph{convexity}  and \emph{scaling}-type assumptions on the L\'evy measure of the noise. The latter `scaling' assumption (see H2 in \cite{Simon} or \eqref{H2} below) is quite restrictive and requires the jump noise to be close to an $\alpha$-stable one, in a sense. This limitation is of a technical nature and is caused by the method used in  \cite{Simon}, \cite{Ishikawa} rather than the problem itself, hence a natural question arises how to remove it in order to get a general support theorem for SDEs with jumps, free from any technical limitations. This question was solved in the one-dimensional setting in  \cite{Simon3} and for canonical (Markus) equations in \cite{Simon4}. For general multidimensional It\^o equations with jumps it apparently requires alternative methods, and remained open since early 2000ies.

In this note, we propose a method for proving a support theorem for general `Type II' SDEs, based on a change of measure and  free from any  non-natural technical limitations. The description we obtain for the topological support for the law of the solutions to  jumping SDEs is of a considerable importance because of its natural applications, in particular, to the study of the \emph{ergodic properties} of the solution to SDE, considered as a Markov process and the \emph{strong maximum principle} for the generator of this process; see more discussion in Section \ref{s23} below.

The structure of the paper is as follows. In Section \ref{s2} the main statement is formulated and provided by a discussion and examples. The proof of the main statement is given in Section \ref{s3}, and to improve readability the  proof of the key lemma (Lemma \ref{l-main}) is postponed  to a separate Section \ref{s4}. For the same reason, proof of a technical estimate \eqref{phi_delta} is given in Appendix \ref{sA}.
\\

\textbf{Acknowledgements.} This research was supported  by the Polish National Science Center grant 2019/33/B/ST1/02923. The author is grateful to Tomas Simon for inspiring discussion and bibliography information.

\section{Main result}\label{s2}

\subsection{Preliminaries} Let $N(\di u,\di t)$ be a Poisson point measure (PPM) on $\R^d\times[0, \infty)$  with the intensity measure $\mu(\di u)\di t$, where $\mu(\di u)$ is a \emph{L\'evy measure}, i.e.
$$
\int_{\R^d}(|u|^2\wedge 1)\mu(\di u)<\infty.
$$
We consider an SDE in $\R^m$
\be\label{SDE}
\di X_t=b(X_t)\, \di t+\int_{|u|\leq 1}c(X_{t-},u) \wt N(\di u,\di t)+ \int_{|u|>1} c(X_{t-},u) N(\di u,\di t), \quad X_0=x_0,
\ee
where $\wt N(\di u,\di t)=\wt N(\di u,\di t)-\mu(\di u)\di t$ is the corresponding compensated PPM.

 The coefficient $c(x,u)$ is assumed to have the form
\be\label{AL}
c(x,u)=\sigma(x)u+r(x,u),
\ee
where $r(x,u)$ is negligible, in a sense, w.r.t. $|u|$ for small values of $|u|$ (see below). In the case $r(x,u)\equiv 0$ equation \eqref{SDE} transforms to
\be\label{SDE_1}
\di X_t=b(X_t)\, \di t+\sigma(X_{t-})\di Z_t, \quad X_0=x_0,
\ee
where the \emph{L\'evy process} $Z$ is given by its It\^o-L\'evy decomposition
$$
\di Z_t=\int_{|u|\leq 1}u \wt N(\di u,\di t)+ \int_{|u|>1} u N(\di u,\di t).
$$

We assume the following.

$\mathbf{H}_1.$ The functions $b:\R^m\to \R^m$ and $\sigma:\R^m\to \R^{m\times d}$ are Lipschitz continuous.

$\mathbf{H}_2.$ There exist constants $C>0, \beta>0$ such that
$$
|r(x,u)-r(y,u)|\leq C|x-y||u|^\beta, \quad |r(x_0,u)|\leq C|u|^\beta,
$$
and
\be\label{moment_mu}
\int_{|u|\leq 1}|u|^\beta\mu(\di u)<\infty.
\ee

If $\beta\leq 1,$ then $\mathbf{H}_2$ yields that the entire function $c(x,u)$ is absolutely integrable w.r.t. $\mu(\di u)$ on the set $\{|u|\leq 1\}$; that is, equation \eqref{SDE} has Type I in the terminology of \cite{Simon}. This (comparatively simple) case  is already studied completely, hence in what follows we consider the case $\beta>1$, only. Note that, in this case  $r(x,u)\leq C(|x-x_0|+1)|u|^\beta\ll |u|$ for small $|u|$, and the linear part $\sigma(x)u$ is the principal one in the decomposition \eqref{AL} for $c(x,u)$.

Recall that the Skorokhod space $\D([0,T], \R^m)$ is the set of c\`adl\`ag functions on $[0, T]$ with the metric
$$
d(f,g)=\sup_{\lambda\in \Lambda_{0,T}}\left(\sup_{0\leq s<t\leq T}\left|\log\frac{\lambda_t-\lambda_s}{t-s}\right|+\sup_{t\in [0,T]}|f(\lambda_t)-g(t)|\right),
$$
where $T>0$ and  $\Lambda_{0,T}$ denotes the set of strictly increasing continuous functions $\lambda:[0, T]\to [0, T]$ such that $\lambda(0)=0, \lambda(T)=T$. It is known that $\D([0,T], \R^m)$ with $d(\cdot, \cdot)$ is a Polish space, e.g. \cite[Section 14]{Bil68}.

Under the assumptions $\mathbf{H}_1$, $\mathbf{H}_2$,  SDE \eqref{SDE} has unique strong solution $X$, see \cite[Theorem~IV.9.1]{IkWa}. We fix \emph{a time horizon} $T>0$, and consider  this solution on the time interval $[0,T]$. The law of this solution in the Skorokhod space $\D([0,T], \R^m)$ will be denoted by $\mathrm{Law}_{x_0,T} (X)$. We aim to describe the support of this law; recall that the (topological) support of a measure $\kappa$ on a metric space $S$ with Borel $\sigma$-algebra is the minimal closed subset $F$ such that $\kappa(S\setminus F)=0$,  this set is denoted by $\supp(\kappa)$. Alternatively, $x\in \supp(\kappa)$ if, and only if,
for any open ball $B$ centered at $x$  one has $\kappa(B)>0$.
\subsection{The main statement}

To describe the support of $\mathrm{Law}_{x_0,T} (X)$ we introduce some constructions. First, we
introduce a kernel $J(x,\di y)$ on $\R^m$ by
$$
J(x, A)=\mu\big(\{u:x+c(x,u)\in A\}\big), \quad A\in \mathcal{B}(\R^m),
$$
and define the set $\mathcal{A}\subset \R^m\times \R^m$ of `admissible jumps' by
$$
(x,y)\in \mathcal{A} \Longleftrightarrow y\in \mathrm{supp}(J(x, \cdot)).
$$
Next, denote by $L$ the set of $\ell\in \R^d$ such that
$$
\int_{|u|\leq 1}|u\cdot\ell| \mu(\di u)<\infty;
$$
here and below we use notation $a\cdot b$ for the scalar product in $\R^d$. It is easy to check that  $L$ is a vector subspace in $\R^d$; we call it the `integrability subspace' for $\mu$ furthermore.  Denote by $u_L$ the orthogonal projection of $u$ on $L$, then
$$
\int_{|u|\leq 1}|u_L| \mu(\di u)<\infty
$$
and under the assumptions $\mathbf{H}_1$, $\mathbf{H}_2$ the following function is well defined and is Lipschitz continuous:
$$
\tilde b(x)=b(x)-\int_{|u|\leq 1}\sigma(x)u_L \mu(\di u)-\int_{|u|\leq 1}r(x,u)\mu(\di u).
$$
Denote by $\mathbf{F}^{\mathrm{const}}_{0,T}$ and $\mathbf{F}^{\mathrm{step}}_{0,T}$ the following classes of functions $f:[0,T]\to \R^d$:
\begin{itemize}
  \item any $f\in \mathbf{F}^{\mathrm{const}}_{0,T}$ takes a constant value in the orthogonal complement $L^\perp$ of the subspace $L$ in $\R^d$;
  \item for any $f\in \mathbf{F}^{\mathrm{step}}_{0,T}$ there exists a partition $[0,T]=[0, \tau_1)\cup[\tau_1,\tau_2)\cup\dots\cup[\tau_j,T]$ such that, on each interval of the partition, $f$ takes a constant value in $L^\perp.$
\end{itemize}
Finally, denote by $\mathbf{S}^{\mathrm{const}}_{0,T,x_0}$ and $\mathbf{S}^{\mathrm{step}}_{0,T,x_0}$ the classes of the solutions to the Cauchy problem for the following piece-wise  ordinary differential equations
\be\label{skeleton}
\phi_t=x_0+\int_0^t\Big(\tilde b(\phi_s)+\sigma(\phi_s)f_s\Big)\, \di s+\sum_{t_k\leq t}\triangle_{t_k}\phi,
\ee
where $\{t_k\}\subset (0,T)$ is an arbitrary finite set,  $f$ is an arbitrary function from the class $\mathbf{F}^{\mathrm{const}}_{0,T}$ or $\mathbf{F}^{\mathrm{step}}_{0,T},$ respectively, and the jumps of the function $\phi$ at the time moments $\{t_k\}$ satisfy
\be\label{admissible}
(\phi_{t_k-}, \phi_{t_k})\in \mathcal{A}\hbox{ for any }k.
\ee
 Denote by $\overline{\mathbf{S}^{\mathrm{const}}_{0,T,x_0}}$ and $\overline{\mathbf{S}^{\mathrm{step}}_{0,T,x_0}}$ the closures of these classes in $\mathds{D}([0,T], \R^m)$.

\begin{theorem}\label{thm_support} Assume $\mathbf{H}_1$, $\mathbf{H}_2$. Then for any $T>0$
$$
\supp\Big(\mathrm{Law}_{x_0,T} (X)\Big)=\overline{\mathbf{S}^{\mathrm{const}}_{0,T,x_0}}=\overline{\mathbf{S}^{\mathrm{step}}_{0,T,x_0}}.
$$
\end{theorem}

 In plain words the  above description of the support can be explained as follows. The SDE \eqref{SDE} contains two parts, the `deterministic flow' part which corresponds to the drift, and the `stochastic jump part' which corresponds to the PPM. These two parts are still related trough the compensator term, which is involved in the stochastic integral w.r.t. $\wt N(\di u, \di t)$. In the simple case where $L=\R^d$, and thus the SDE is of the Type I, these parts can be  completely separated, and \eqref{SDE} can be written in the form
\be\label{SDE_2}
\di X_t=\wt b(X_t)\, \di t+\int_{\R^d}c(X_{t-},u)  N(\di u,\di t),
\ee
where the \emph{effective drift} coefficient is defined by $\wt b(x)=b(x)-\int_{|u|\leq 1}c(x,u)\, \mu(\di u)$. In this case \eqref{skeleton} does not contain the part with $f$ because $L^\perp=\{0\}$, and the description of the support of the law of the solution $X$ is intuitively clear: the solution  follows the deterministic flow defined by the effective drift, then makes a jump, admissible for the stochastic jump part,   then again follows the deterministic flow, etc. This description was thoroughly proved in \cite[Theorem~I]{Simon}. The general (Type II) case is more sophisticated, since the compensator term cannot be separated from the stochastic integral. Theorem \ref{thm_support} actually tells us that the above description of the support remains essentially correct, with the following two important changes:
\begin{itemize}
  \item the effective drift involves only the parts of $c(x,u)$ which are absolute integrable, i.e. $r(x,u)$ and $\sigma(x) u_L$;
  \item the `non-integrability subspace' $L^\perp$ induces an extra drift part, which may act at \emph{arbitrary} direction of the form $\sigma(x) h, h\in L^\perp.$
\end{itemize}
Such a description of the support was
  conjectured by T. Simon, see \cite[Remark 3.3(c)]{Simon}, by analogy with a similar result proved for L\'evy processes in \cite{Simon2}. Theorem \ref{thm_support} proves this conjecture in a wide generality, i.e. without any specific assumptions on the SDE except the natural conditions $\mathbf{H}_1$, $\mathbf{H}_2$ which yield strong existence of the solution.

\begin{remark} Our description of admissible jumps differs from the one adopted in \cite{Simon}, which requires that
$$
y=x+c(x,u), \quad u\in \mathrm{supp}(\mu).
$$
These two descriptions coincide for $c(x,u)$  continuous in $u$; for discontinuous $c(x,u)$ the latter one is no longer applicable. To see this, one can consider a simple example where $m=d=1,$ $\mu$ is  a finite  discrete measure
 $\mu$ which has a full support in $\R$, $m=1$, and $c(x,u)=1_K(u)+21_{\R\setminus K}(u)$, where $K$ is a countable dense subset in $\R$ such that $\mu(\R\setminus K)=0$.
\end{remark}
\subsection{Discussion: applications and examples}\label{s23}\phantom{}\\
 Support theorems are involved naturally in the study of ergodic properties of the Markov process associated to the SDE; namely, they provide a natural tool for proving that the process is topologically irreducible, see  \cite[Proof of Proposition~5.3]{HMS11} or \cite[Theorem~1.3~ and~Section~4.3]{Kulik}. This gives a natural application field for Theorem \ref{thm_support}. Namely, denote by $S_T(x_0)$ the closure of the set of the values $\phi_T$, where $\phi$ runs through the solutions to \eqref{skeleton} with arbitrary finite set
 $\{t_k\}\subset (0,T)$, jumps of $\phi$ satisfying \eqref{admissible}, and  arbitrary function $f$  from the class $\mathbf{F}^{\mathrm{step}}_{0,T}.$ By Theorem \ref{thm_support}, the solution to \eqref{SDE} satisfies 
$$
\supp\Big(\mathrm{Law}\, (X_T)\Big)=S_T(x_0).
$$
Thus, in order to provide the topological irreducibility for the Markov process associated to \eqref{SDE}, it is sufficient to show that
 \be\label{full}
 S_T(x)=\R^m\hbox{  for all }x\in \R^m.
 \ee

Another application of the support theorems dates back to the original paper by Stroock and Varadhan \cite{StroockWaradhan}, and concerns \emph{the strong maximal principle} for the generator $\mathscr{L}$ of the Markov process $X$, which states that any sub-harmonic function for $\mathscr{L}$ which reaches its maximum on the given set is constant on this set. We refer an interested reader for a detailed discussion to \cite[Theorem IV.8.3]{IkWa} or  \cite[Remark 3.3(b)]{Simon}. Here, we mention briefly the following simple corollary: if $\bigcup_{T>0} S_T(x)$ is dense in $\R^m$ for any $x\in \R^m$, then
the strong maximal principle for $\mathscr{L}$ holds true on the entire $\R^m$.  For this property to hold, it is clearly sufficient that \eqref{full} holds true for some $T>0$.

  Below, we give two simple
sufficient conditions where \eqref{full} holds true for any $T>0$.

\begin{example}\label{ex1} \emph{(Jump noise satisfying the `cone condition').}
  Let  there exist $\theta\in (0,1)$ such that, for any $\ell\in \R^d, |\ell|=1$ and $\eps>0$, the intersection of the cone $\{u:u\cdot \ell\geq\theta |u|\}$ with the ball $\{|u|< \eps\}$ has positive measure $\mu$. Assume also that $\sigma(x)$ is \emph{point-wise degenerate}, i.e. $\mathrm{rank}\, \sigma(x)=m\leq d, x\in \R^m$. Then $S_T(x)=\R^m$ for all $x\in \R^m, T>0$. To prove this assertion, one can take $f\equiv 0$ and organize for given $x,y\in \R^m, \eps>0$ and $T>0$  sequences of (frequent) jump times $\{t_k\}$ and (small) jumps amplitudes $u_k$ which would force the solution to \eqref{skeleton} with $x_0=x$ to take the final value $y$ with $|\phi_T-y|<\eps$. This construction is essentially the same as in the proof of \cite[Proposition~4.17]{Kulik}, thus we omit the details here.
\end{example}
\begin{example}\label{ex2} \emph{(Jump noise of the `strong Type II').} Let there be no integrability directions for $\mu(\di u)$; that is, $L=\{0\}$.
   Assume also that  $\sigma(x)$ is point-wise degenerate. Then $S_T(x)=\R^m$ for all $x\in \R^m, T>0$. To prove this assertion, one can ignore the jump part in \eqref{skeleton} and consider the solution to the SDE
   $$
   \di \phi_t=\tilde b(\phi_t)\, \di t+\sigma(\phi_t)\, f_t\di t, \quad \phi_0=x.
   $$
   Since $f$ can be arbitrary piece-wise constant function $[0,T]\to \R^d$ and $\sigma(x)$ is point-wise degenerate, by a proper choice of $f$ respective solution $\phi$ can be made arbitrarily close to
   $$
   \phi_t^{x,y,T}=x+\frac{t}{T}(y-x).
   $$
  \end{example}

 Let us give two more particular  examples illustrating the range of applicability of  Theorem \ref{thm_support} and sufficient conditions from Example \ref{ex1} and Example \ref{ex2}. For that purpose, we recall the
 auxiliary `scaling' assumption imposed in \cite{Simon}, \cite{Ishikawa}: for some $\alpha\in (0,2)$,
\be\label{H2}
\int_{|u|\leq \eps}(u\cdot \ell)^2\mu(\di u)\asymp \eps^{2-\alpha}, \quad |\ell|=1.
\ee

The following  example concerns the \emph{cylindrical L\'evy processes}, which have been studied extensively in the last decades, e.g. \cite{KR17}, \cite{KR19}, \cite{PZ}.

\begin{example}\label{ex3} Let $\mu(\di u)$ be the L\'evy measure of the  $Z=( Z^1, \dots, Z^d)$ with the independent components $Z^i, i=1, \dots, d$. Let the components be symmetric $\alpha_i$-stable processes on $\R$, then the scaling assumption \eqref{H2} fails unless all the stability indices $\alpha_i, i=1, \dots, d$ are the same; that is,  \cite[Theorem~II]{Simon} is not applicable. If at least one $\alpha_i>1$, then this is the Type II L\'evy noise and \cite[Theorem~I]{Simon} is not applicable, as well. On the other hand, Theorem \ref{thm_support} can be applied  assuming the coefficients satisfy $\mathbf{H}_1$, $\mathbf{H}_2$; note that $\beta$ can be taken arbitrary $>\max_i\alpha_i.$ The L\'evy measure $\mu(\di u)$ is supported by the collection of the coordinate axes in $\R^d$ and thus satisfies the `cone condition' from Example \ref{ex1}. Hence, if   $\sigma(x)$ is point-wise degenerate, then  identity \eqref{full} holds true for any $T>0$.
\end{example}

Our second example shows that L\'evy measure can be of a `strong Type II' even if even the (rather mild) `cone condition' fails.

  \begin{example}\label{ex4} Let  the L\'evy measure $\mu(\di u)$ on $\R^2$  be the image of the symmetric $\alpha$-stable L\'evy measure $\frac{\di z}{|z|^{1+\alpha}}$ on $\R$
  under the mapping
  $$
  z\mapsto (z,|z|^\gamma\mathrm{sgn}\, z)
  $$
  with some $\gamma>1$. Then the scaling condition \eqref{H2} fails, and for $\alpha>1$ (i.e. when the noise is of Type II) the previous results are not applicable. On the other hand, Theorem \ref{thm_support} applies. Note that the L\'evy measure $\mu(\di u)$ is quite degenerate and the `cone condition' fails. On the other hand, for $\gamma\leq \alpha$ the noise has strong Type II.  Hence, if $\sigma(x)$ is point-wise degenerate, then identity \eqref{full} holds true for any $T>0$.
\end{example}

\section{Proof of Theorem \ref{thm_support}}\label{s3}

To prove the required statement, it is sufficient to prove the following two inclusions:
\be\label{inclusions} \supp\Big(\mathrm{Law}_{x_0,T} (X)\Big)\subset\overline{\mathbf{S}^{\mathrm{const}}_{0,T,x_0}}, \quad \overline{\mathbf{S}^{\mathrm{step}}_{0,T,x_0}}\subset\supp\Big(\mathrm{Law}_{x_0,T}(X)\Big).
\ee
The proof of the first inclusion is simple and standard, corresponding argument was explained in \cite{Simon}; for the reader's convenience we outline the argument here. Consider a family of SDEs
\be\label{SDE_eta}
\di X_t^\eta=b_\eta(X_t^\eta)\, \di t+\int_{\eta\leq |u|\leq 1}c(X_{t-}^\eta,u) \wt N(\di u,\di t)+ \int_{|u|>1} c(X_{t-}^\eta,u) N(\di u,\di t), \ee
with $X_0^\eta=x_0$, where
$$
b_\eta(x)=b(x)-\int_{|u|<\eta}\sigma(x)u_L \mu(\di u)-\int_{|u|<\eta}r(x,u)\mu(\di u).
$$
It is easy to check that $b_\eta\to b, \eta\to 0$ uniformly on compact subsets of $\R^m$. Then by the usual stochastic calculus technique one can show that
\be\label{eta_conv}
\sup_{t\in [0,T]}|X_t^\eta-X_t|\to 0, \quad \eta\to 0
\ee
in probability; see the proof of a similar statement in Lemma \ref{l3} below.  On the other hand, equation \eqref{SDE_eta} can be written in the form
$$
\di X_t^\eta=\wt b(X_t^\eta)\, \di t+\sigma(X_t^\eta)\upsilon_\eta\, \di t+\int_{|u|\geq \eta}c(X_{t-}^\eta,u) N(\di u,\di t),
$$
where
\be\label{shift}
\upsilon_\eta=\int_{\eta\leq |u|\leq 1}(u-u_L)\, \mu(\di u)\in L^\perp.
\ee
The PPM $N$, restricted to $\{|u|\geq \eta, t\in [0, T]\}$, has a finite set of atoms (`jumps'). Then the solution to the latter equation can be represented path-wise as a collection of solutions to ODEs of the form \eqref{skeleton}, where  $f\equiv \upsilon_\eta$ belongs to
$\mathbf{F}^{\mathrm{const}}_{0,T}$, $\{t_k\}$ are equal to the time instants of the jumps, and  $\triangle_{t_k}\phi=c(\phi_{t_{k}-}, u_k)$, where  $\{u_k\}$ are equal to the amplitudes of the jumps for $N$. Note that, for any $t_k$, the value  $\phi_{t_{k}}$ is an image of $x=\phi_{t_{k-}}, u=u_k$ under the mapping
$$
(x,u)\mapsto x+c(x,u),
$$
and thus the pair $
(\phi_{t_{k}-}, \phi_{t_{k}})$ is admissible. Therefore, for any $\eta$,
$$
\supp\Big(\mathrm{Law}_{x_0,T} (X^\eta)\Big)\subset\overline{\mathbf{S}^{\mathrm{const}}_{0,T,x_0}}\Longleftrightarrow\P(X^\eta|_{[0,T]}\in \overline{\mathbf{S}^{\mathrm{const}}_{0,T,x_0}})=1.
$$
By \eqref{eta_conv}, the laws of $X^\eta$ in $\D([0,T], \R^m)$ weakly converge to the law of $X$, which gives for a closed set  $\overline{\mathbf{S}^{\mathrm{const}}_{0,T,x_0}}$
$$
\P(X|_{[0,T]}\in \overline{\mathbf{S}^{\mathrm{const}}_{0,T,x_0}})\geq \limsup_{\eta\to 0}\P(X^\eta|_{[0,T]}\in \overline{\mathbf{S}^{\mathrm{const}}_{0,T,x_0}})=1,
$$
and completes the first inclusion in \eqref{inclusions}.

The second inclusion is the main part of the theorem.
In order to proceed with its proof, we re-write the original SDE \eqref{SDE} to the form
\be\label{SDEa}\begin{aligned}\di X_t=\wt b_\eta (X_t)\, \di t&+\sigma(X_t)\upsilon_\eta\, \di t+\int_{|u|<\eta}c(X_{t-},u) \wt N(\di u,\di t)
\\&+ \int_{|u|\geq \eta}c(X_{t-},u) N(\di u,\di t),
\end{aligned}
\ee
where $\eta>0$ is a (small) parameter which is yet to be chosen and
$$
\wt b_\eta(x)=b(x)-\int_{\eta\leq |u|\leq 1}\sigma(x)u_L \mu(\di u)-\int_{\eta\leq |u|\leq 1}r(x,u)\mu(\di u),
$$
recall that $\upsilon_\eta$ is given by \eqref{shift}. Consider SDE
\be\label{SDEb}\di X_t^{\eta, \mathrm{trunc}}=\wt b_\eta (X_t^{\eta, \mathrm{trunc}})\, \di t+\sigma(X_t^{\eta, \mathrm{trunc}})\upsilon_\eta\, \di t+\int_{|u|<\eta}c(X_{t-}^{\eta, \mathrm{trunc}},u) \wt N(\di u,\di t),
\ee
which can be seen as a modification of  \eqref{SDEa} with `large jumps' being truncated, i.e. with the PPM $N(\di u,  \di t)$ changed to its restriction to $\{|u|<\eta\}\times [0,T]$. We will denote by $X^{x, S,\eta, \mathrm{trunc}}$ the solution of \eqref{SDEb} with $X_S=x$.

Let $f\in  \mathbf{F}^{\mathrm{step}}_{0,T}, x\in \R^m, S\in [0, T]$ be fixed and $\phi^{x, S, f}$ be the solution of the ODE
\be\label{ODE}
\phi_t=x+\int_S^t\Big(\widetilde b(\phi_s)+\sigma(\phi_s)f_s\Big)\, \di s, \quad t\in [S, T]
\ee
The following statement is the cornerstone of the entire proof. In what follows  we denote by $B(x, r)$ the open ball in $\R^m$ with the center $x$ and radius $r$.
\begin{lemma}
  \label{l-main}(The Key Lemma)  There exists $\rho\in (0,1)$ such that, for any given $f\in  \mathbf{F}^{\mathrm{step}}_{0,T}, x\in \R^m,$ and $\gamma>0,$ there exists $\eta^{f,x, \gamma}>0$ such that
$$\ba
p^{\mathrm{trunc}}(\eta, f, x, \gamma)
:=
\inf_{x'\in B(x,\rho\gamma), 0\leq S\leq Q\leq T}\Pp\left(\sup_{t\in [S, Q]}|X_t^{x',S,\eta, \mathrm{trunc}}-\phi_t^{x,S, f}|\leq \gamma\right)>0
\ea$$
for any $\eta\in (0,\eta^{f,x, \gamma}].$
\end{lemma}

We postpone the proof of  Lemma \ref{l-main} to a separate Section \ref{s4}; here we explain the   argument (quite standard, though a bit cumbersome) which provides the second inclusion in \eqref{inclusions} once this key lemma  is proved.

Fix $\eta>0$ and decompose
$$
N(\di u, \di t)=1_{|u|<\eta}N(\di u, \di t)+1_{|u|\geq \eta}N(\di u, \di t)=:N_\eta(\di u, \di t)+N^\eta(\di u, \di t).
$$
The PPM $N^\eta(\di u, \di t)$ a.s. has a finite number of atoms with $t\in[0, T]$; say, $\{(\xi_j,\tau_j)\}_{j=1}^J.$
It is well known that the PPMs $N_\eta(\di u, \di t)$ and $N^\eta(\di u, \di t)$ are independent. In addition,  the random variable
$$
J=N(\{|u|\geq \eta\}\times [0,T])
$$
has the Poisson distribution with the intensity $T\mu(|u|\geq \eta),$ and conditioned by the event $\{J=K\},$ the laws of the random vectors $\{\xi_j\}_{j=1}^K, \{\tau_j\}_{j=1}^K$ are independent. These laws are  the $K$-fold product of
$$
\P^\eta(\di u):=\frac{1}{\mu(\{v:|v|\geq \eta\})}\mu(\di u)
$$
for $\{\xi_j\}_{j=1}^K,$ and the uniform distribution on  the simplex
$$
\Delta_K(0, T):=\{(s_1, \dots, s_K): 0\leq s_1\leq \dots \leq s_K\leq T\}
$$
for $\{\tau_j\}_{j=1}^K.$ With this information in mind, it is easy to show that $X$ has a positive probability to stay in any neighbourhood of a given function $\phi\in\mathbf{S}^{\mathrm{step}}_{0,T,x_0}$.

Namely, let $\phi\in\mathbf{S}^{\mathrm{step}}_{0,T,x_0}$ be fixed with the  corresponding  function $f\in \mathbf{F}^{\mathrm{step}}_{0,T}$ and points $\{t_k, k=1, \dots, K\}$ from the equation   \eqref{skeleton} for $\phi$. Denote $x_k=\phi_{t_k-}, y_k=\phi_{t_k}$; we can and will assume that $x_k\not=y_k$ because otherwise  we can simply exclude the point $t_k$ from the equation \eqref{skeleton}. Denote also $t_0=0, t_{K+1}=T$, then for any positive
$$
\delta<\delta_{\{t_k\}}:=\frac12\min_{k=0, \dots, K}|t_{k+1}-t_k|
$$
we have
$$
p(\eta, \delta, \{t_k\}):=\P(J=K, |t_k-\tau_k|<\delta, k=1, \dots,K)=e^{-T\mu(|u|\geq \eta)}\Big(2\delta\mu(|u|\geq \eta)\Big)^K,$$
which is positive.

Process $X$ follows the truncated SDE \eqref{SDEb} between the `big jumps instants' $\tau_j$, and at these instants satisfies
$$
\triangle_{\tau_j} X=c(X_{\tau_j}, \xi_{j}), \quad j=1, \dots, J.
$$
Denote by $X^{s_1, \dots, s_K}$ the similar process with $J=K$ and $\tau_k=s_k$: it
follows the truncated SDE \eqref{SDEb} between  $s_k$, and at these time instants satisfies
$$
\triangle_{s_k} X=c(X_{s_k}, \xi_{k}), \quad k=1, \dots, K;
$$
the random vector $\{\xi_j\}_{j=1}^K$ has the law $(\P^\eta)^{\otimes K}$.
Then
\be\label{sk}\ba
\P(X\in A)&\geq \P(X\in A, J=K)
\\&=\int_{\Delta_K(0, T)}\P(X^{s_1, \dots, s_K}\in A)\P(J=K, (\tau_1, \dots\tau_k)\in ds_1\dots ds_K)
\\&\geq p(\eta, \delta, \{t_k\})\inf_{|t_k-s_k|<\delta, k=1, \dots,K}\P(X^{s_1, \dots, s_K}\in A).
\ea\ee
We will apply this inequality with
$$
A=\{x(\cdot): d(x(\cdot), \phi)\leq  \eps \},
$$
in this case $\P(X^{s_1, \dots, s_K}\in A)$ can be bounded from below  as follows. Denote $s_0=0$, $s_{K+1}=T$, and $y_0=x_0$
Define a function
$\phi^{s_1, \dots, s_K}$ as follows: on each of the intervals $[s_k, s_{k+1}), k=0, \dots, K$ it satisfies \eqref{ODE} with $x=y_k, S=s_k$; it is also continuous at the end point $T=s_{K+1}$.
Recall that the initial function $\phi_t$ follows similar description with $\{t_k\}$ instead of $\{s_k\}$. Using this observation, it is easy to show that
\be\label{phi_delta}
\sup_{{|t_k-s_k|<\delta, k=1, \dots,K}} d(\phi^{s_1, \dots, s_K}, \phi)\to 0, \quad \delta\to 0;
\ee
for the reader's convenience we prove this relation in Appendix \ref{sA}. Hence we can fix $\delta(\eps)>0$ such that
$$
\sup_{{|t_k-s_k|<\delta, k=1, \dots,K}} d(\phi^{s_1, \dots, s_K}, \phi)\leq \frac\eps2, \quad \delta\in (0, \delta(\eps)),
$$
and thus for any  $\{s_k\}$ with $|t_k-s_k|<\delta, k=1, \dots,K$
\be\label{eps/2}
\P(d(X^{s_1, \dots, s_K}, \phi)\leq \eps)\geq \P\left(d(X^{s_1, \dots, s_K}, \phi^{s_1, \dots, s_K})\leq \frac\eps2\right).
\ee
Next, denote
$$
\mathcal{F}_t=\sigma(X^{s_1, \dots, s_K}_s, s\leq t)=\sigma(N_\eta([0,s]\times du), s\leq t, \{\xi_k, s_k\leq t\}),
$$
and
$$
\mathcal{G}_t=\mathcal{F}_{t-}=\sigma(N_\eta([0,s]\times du), s\leq t, \{\xi_k, s_k<t\}),
$$
note that $\xi_k$ is independent on $\mathcal{G}_{s_k}$ for any $k$.
The following holds:
 \begin{itemize}
   \item[(I)] For any $k=0, \dots, K$ and $\gamma>0$, the conditional probability w.r.t. $\mathcal{F}_{s_k}$ for the event
   $$
   \{|X^{s_1, \dots, s_K}_t-\phi_t^{s_1, \dots, s_K}|\leq  \gamma, t\in [s_k, s_{k+1})\}
   $$
   equals
   $$
   \P\left(\sup_{t\in [s_k, s_{k+1})}|X_t^{x, s_k,\eta, \mathrm{trunc}}-\phi_t^{y_k,s_k, f}|\leq \gamma\right)\Big|_{x=X_{s_k}^{s_1, \dots, s_K}},
   $$
   and
   by Lemma \ref{l-main} is bounded from below by $p^{\mathrm{trunc}}(\eta, f,  y_k, \gamma)$ on the set
   $$
   \left\{|X^{s_1, \dots, s_K}_{s_k}-y_k|<\frac\gamma2\right\}\in \mathcal{F}_{s_k},
   $$
   provided that $\eta\leq\eta^{f,y_k,\gamma}$;
   \item[(II)] for any $k=1, \dots, K$ and $\gamma>0$,  the conditional probability w.r.t. $\mathcal{G}_{s_k}$ for the event
   $$
   \left\{|X^{s_1, \dots, s_K}_{s_k}-y_k|<\frac\gamma2\right\}
   $$
   equals
   $$\ba
   P^\eta&\left(|x+c(x,\xi)-y_k|<\frac\gamma2\right)\Big|_{x=X^{s_1, \dots, s_K}_{s_k-}}
   \\&=\frac{1}{\mu(\{u:|v|\geq \eta\})}\mu\left(\left\{u:|u|\geq \eta, |x+c(x,u)-y_k|<\frac\gamma2\right\}\right)\Big|_{x=X^{s_1, \dots, s_K}_{s_k-}}.
   \ea$$
   \end{itemize}

Recall that each pair $(x_k, y_k)=(\phi_{t_k-}, \phi_{t_k})$  is admissible, hence for any $\gamma>0$ and $k$,
$$
J(x_k,B(y_k, \gamma))=\mu(\{u:x_k+c(x_k,u)\in B(y_k, \gamma)\})>0.
$$
Take
$$
\gamma_*=\frac13\min_{k}|x_k-y_k|,
$$
then by  the assumptions $\mathbf{H}_1$, $\mathbf{H}_2$ there exists $\eta^{*}>0$ such that, for all $k$,
\be\label{trunc}
|u|\leq \eta^*\Longrightarrow |c(x_k, u)|< |x_k-y_k|-{\gamma_*} \Longrightarrow x_k+c(x_k, u)\not \in B\left(y_k, {\gamma_*}\right).
\ee
This yields for any $\gamma\in (0, \gamma_*]$
$$
\mu(\{u:|u|>\eta^*, x_k+c(x_k, u) \in B(y_k, \gamma)\})>0, \quad k=1, \dots, K.
$$
Moreover, because $c(x,u)$ is continuous w.r.t. $x$ we have by the usual weak continuity arguments that, for each $\gamma\in (0, \gamma_*]$ and $k$,  there exists $\gamma'>0$ such that
\be\label{trunc_prim}
\inf_{x\in B(x_k,2\gamma')}
\mu\left(\left\{u:|u|>\eta^*, x+c(x, u) \in B\left(y_k, \gamma\right)\right\}\right)>0.
\ee
Now we can finalize the construction and the estimate. Define iteratively $\gamma_{k}, \gamma_k'$ for $k=K, \dots, 1$ as follows:
Take
$$
\gamma_{K}=\min\left(\frac{\rho\eps}2,\gamma_*\right)
$$
(with $\rho$ given by Lemma \ref{l-main}) and define $\gamma'_K$ such that
\eqref{trunc_prim} holds true with $\gamma=\gamma_K$. Once $\gamma_{k+1}, \gamma_{k+1}'$ are defined, take
$$
\gamma_{k}=\min\left(\rho\gamma_{k+1}',\gamma_{k+1}\right)
$$
and define $\gamma'_k$ such that
\eqref{trunc_prim} holds true with $\gamma=\gamma_{k+1}$.

It follows from the calculations in Appendix \ref{sA} that $\delta\in (0, \delta(\eps))$ can be taken small enough such that, for any  $s_1, \dots, s_K$ with $|t_k-s_k|<\delta, k=1, \dots, K$,
\be\label{gamma}
|\phi_{s_k-}^{s_1, \dots, s_K}-x_k|<\gamma_{k}'.
\ee
We fix such $\delta>0$ and the truncation level
$$
\eta=\min\{\eta_*, \eta^{f,y_k,\gamma_{k+1}'}, k=0, \dots, K \}.
$$
Denote
$$
A_k=\left\{\sup_{t\in [s_k, s_{k+1})}|X^{s_1, \dots, s_K}_t-\phi_t^{s_1, \dots, s_K}|\leq \gamma_{k+1}'\right\}, \quad k=0, \dots, K-1,
$$
$$A_K=\left\{\sup_{t\in [s_K, T]}|X^{s_1, \dots, s_K}_t-\phi_t^{s_1, \dots, s_K}|\leq \frac\eps2\right\},
$$
and
$$
B_k=\{|X^{s_1, \dots, s_K}_{s_k}-y_k|<\gamma_k\}, \quad k=1, \dots, K.
$$
Then $ A_k\in \mathcal{G}_{s_{k+1}-}, B_k\in \mathcal{F}_{s_k}, k=1,\dots, K$, and we have the following:

\begin{itemize}
  \item Since $\gamma_k\leq \rho\gamma_k'$ for any $k=1, \dots, K-1$, by Lemma \ref{l-main}  we have
  $$
  \P(A_k|\mathcal{F}_{s_{k}})\geq  p^{\mathrm{trunc}}(\eta, f,  y_k, \gamma_{k+1}') \hbox{ a.s. on the set }B_k.
  $$
  We also have by Lemma \ref{l-main}
  $$
   \P(A_1)\geq  p^{\mathrm{trunc}}(\eta, f,  y_0, \gamma_{1}')
   $$
   and
   $$
  \P(A_K|\mathcal{F}_{s_{K}})\geq  p^{\mathrm{trunc}}\left(\eta, f,  y_K, \frac\eps2\right) \hbox{ a.s. on the set }B_K.
  $$
  \item By \eqref{gamma},  for any $k=1, \dots, K$ we have  $|X^{s_1, \dots, s_K}_{s_k-}-x_k|<2\gamma_k'$ on the set $A_{k-1}$. Then
  $$\ba
   \P(B_k|\mathcal{G}_{s_{k}})&\geq  p^{\mathrm{jump}}_k(\eta)
   \\&:=
\frac{1}{\mu(\{u:|v|\geq \eta\})}\inf_{x\in B(x_k,2\gamma'_k)}
\mu\left(\left\{u:|u|>\eta, x+c(x, u) \in B\left(y_k, \gamma_k\right)\right\}\right),
  \ea$$
  and since $\eta\leq \eta_k, k=1, \dots, K$ we have by \eqref{trunc_prim}
  $$
  p^{\mathrm{jump}}_k(\eta)>0, \quad k=1, \dots, K.
      $$
\end{itemize}

By the construction, we have  $\gamma_k, \gamma_k'\leq \frac\eps2$ for all $k=1, \dots, K$. Then by the telescopic property of the conditional expectations we have for any $s_1, \dots, s_K$ with $|t_k-s_k|<\delta, k=1, \dots, K$
$$\ba
\P&\left(\sup_{t\in [0, T]}|X^{s_1, \dots, s_K}_{t}-\phi^{s_1, \dots, s_K}_{t}|\leq \frac\eps2\right)\geq \P(A_0\cap B_1\cap A_1\dots\cap B_K\cap B_K)
\\&\hspace*{1cm}\geq  p^{\mathrm{trunc}}\left(\eta, f,  y_K, \frac\eps2\right)\prod_{k=1}^{K}\Big(p^{\mathrm{trunc}}(\eta, f,  y_{k-1}, \gamma_{k}') p^{\mathrm{jump}}_k(\eta)\Big)=:q(\eta, \delta)>0.
\ea
$$
Then by \eqref{eps/2}
$$\ba
\P(d(X^{s_1, \dots, s_K}, \phi)\leq \eps)&\geq \P\left(d(X^{s_1, \dots, s_K}, \phi^{s_1, \dots, s_K})\leq \frac\eps2\right)
\\&\geq \P\left(\sup_{t\in [0, T]}|X^{s_1, \dots, s_K}_{t}-\phi^{s_1, \dots, s_K}_{t}|\leq \frac\eps2\right)
\\&\geq q(\eta, \delta)
\ea$$
and by
\eqref{sk}
$$
\P(d(X, \phi)\leq \eps)\geq  p(\eta, \delta, \{t_k\}) q(\eta, \delta).
$$
In these estimates, the choice of $\delta, \eta$ depends on $\phi\in \mathbf{S}^{\mathrm{step}}_{0,T,x_0}$ and $\eps>0$, only, hence the proof of
\eqref{inclusions} is complete.

\section{Proof of the Key Lemma}\label{s4} We begin the proof of Lemma \ref{l-main} with the following auxiliary result.

\begin{lemma}\label{l1}
  For any $w\in L^\perp$ and $\eta>0$ there exist $\zeta\in (0, \eta)$ and a function $g:\R^d\to [-\frac12, \frac{1}{2}]$ such that
$
g(u)=0\hbox{ whenever either }|u|\leq \zeta\hbox{ or }|u|\geq \eta
$
and
$$
\int_{\R^d}(u-u_L)g(u)\, \mu(\di u)=w.
$$
\end{lemma}
\begin{proof} For a given $0<\zeta<\eta$, denote by $G_\zeta^\eta$ the set of all functions $g:\R^d\to [-\frac12, \frac{1}{2}]$ such that
$
g(u)=0\hbox{ whenever either  }|u|\leq \zeta\hbox{ or }|u|\geq \eta.
$ This set is convex and symmetric; hence
$$
V_\zeta^\eta=\left\{\int_{\R^d}(u-u_L)g(u)\, \mu(\di u), g\in G_\zeta^\eta\right\}
$$
is a symmetric convex subset of $L^\perp$, and so is the set
$$
V_0^\eta:=\bigcup_{\zeta\in (0, \eta)}V_\zeta^\eta
$$
The statement of the lemma is equivalent to
\be\label{id}
V_0^\eta= L^\perp.
\ee
Assuming \eqref{id} to fail, we have that $V_0^\eta$ is a proper symmetric convex subset of $L^\perp,$ and thus there exist $\ell\in L^\perp\setminus\{0\}$ and $c>0$ such that
\be\label{ass}
-c\leq v\cdot \ell\leq c, \quad v\in V_0^\eta;
\ee
recall that  $v\cdot \ell$ denotes the scalar product in $\R^d$. It follows from \eqref{ass} that, for every $\zeta\in (0, \eta)$ and $g\in G_\zeta^\eta$,
$$\ba
\left|\int_{\R^d}\ell\cdot u\,g(u)\, \mu(\di u)\right|&=\left|\int_{\R^d}\ell\cdot(u-u_L)g(u)\, \mu(\di u)\right|
\\&=\left|\ell\cdot\int_{\R^d}(u-u_L)g(u)\, \mu(\di u)\right|\leq c,
\ea
$$
in the second identity we have used that $u_L\in L$ is orthogonal to $\ell\in L^\perp$. Taking
$$
g_\zeta^\eta(u)=\frac12\mathrm{sign}\,(\ell\cdot u)1_{\zeta<|u|<\eta},
$$
we get from the previous inequality that
$$
\int_{\zeta<|u|<\eta}|\ell\cdot u|\, \mu(\di u)\leq 2c, \quad \zeta\in (0, \eta),
$$
and passing to the limit as $\zeta\to 0$ we obtain
$$
\int_{|u|<\eta}|\ell\cdot u|\, \mu(\di u)\leq 2c<+\infty.
$$
This means that $\ell\in L$, which contradicts to the fact that $\ell\in L^\perp\setminus\{0\}$. This contradiction shows that $V_0^\eta$ should coincide with entire $L^\perp,$ which completes the proof.
\end{proof}

By Lemma \ref{l1}, for a fixed $f\in  \mathbf{F}^{\mathrm{step}}_{0,T}$ and $\eta>0$  one can choose a function $g_t^{f, \eta}(u)$ such that it is a step-wise function of $t$, takes values in $[-\frac12, \frac{1}{2}]$,  satisfies
\be\label{star}
\int_{\R^d}(u-u_L)g_t^{f, \eta}(u)\, \mu(\di u)=\upsilon_\eta-f_t, \quad t\in [0, T]
\ee
with  $\upsilon_\eta\in L^\perp$  given by \eqref{shift}, and,  for some $\zeta^{f, \eta}>0$, one has  $g_t^{f, \eta}(u)=0$ whenever either $|u|\geq \eta$ or $|u|\leq \zeta^{f, \eta}$.

 We write the compensated PPM in \eqref{SDEb} in the form $N(\di u, \di t)-\mu(\di u)\di t$ and
consider the same SDE with another PPM $Q^{f, \eta}(\di u, \di t)$ which has the intensity measure $(1+g_t^{f, \eta}(u))\mu(\di u)\di t$:
\be\label{SDEc}\di Y_t^\eta=\wt b_\eta (Y_t^\eta)\, \di t+\sigma(Y_t)\upsilon_\eta\, \di t+\int_{|u|<\eta}c(Y_{t-}^\eta,u) (Q^{f, \eta}(\di u,\di t)-\mu(\di u)\di t).
\ee
Similarly to the notation used in Lemma \ref{l-main}, we denote by $Y_t^{x,S,\eta}, t\geq S$ the solution to \eqref{SDEc} with $Y^\eta_S=x$. The following lemma is based on quite standard stochastic calculus estimates.
\begin{lemma}\label{l3} There exists $\rho\in (0,1)$ such that, for arbitrary $\gamma\in (0,1], x\in \R^d$,
$$
\inf_{x'\in B(x, \gamma\rho), 0\leq S\leq Q\leq T}\Pp\left(\sup_{t\in [S, Q]}|Y_t^{x',S,\eta}-\phi_t^{x,S, f}|\leq \gamma\right)\to 1, \quad \eta\to 0.
$$
\end{lemma}
\begin{proof} For simplicity of notation we take $S=0, Q=T$ and omit the index $S$. We also consider the scalar case $d=1$; for $d>1$ similar  estimates should be performed coordinate-wise.

 The stochastic integral part in the equation \eqref{SDEc} can be written as
$$
\int_{|u|<\eta}c(Y_{t-},u) \wt Q^{f, \eta}(\di u,\di t)-\int_{|u|<\eta}c(Y_{t-},u) g_t^{f, \eta}(u)\mu(\di u)\di t,
$$
hence taking \eqref{star} into account  we can write this equation in the form
\be\label{SDEd}\di Y_t^\eta=\wt b_t^{f, \eta} (Y_t^\eta)\, \di t+\sigma(Y_t^\eta)f_t\, \di t+\int_{|u|<\eta}c(Y_{t-}^\eta,u) \wt Q^{f, \eta}(\di u,\di t),
\ee
where
$$\ba
\wt b_t^{f, \eta}(x)&= \wt b_\eta(x)-\int_{|u|<\eta} \sigma(x)u_L g_t^{f, \eta}(u)\mu(\di u)-\int_{|u|<\eta} r(x,u) g_t^{f, \eta}(u)\mu(\di u)
\\&=\wt b(x)+\int_{|u|<\eta} \sigma(x)u_L (1-g_t^{f, \eta} (u))\mu(\di u)+\int_{|u|<\eta} r(x,u) (1-g_t^{f, \eta}(u))\mu(\di u).
\ea $$
Since the functions $u_L, |u|^\beta$ are integrable w.r.t. $\mu(du)$ on $\{|u|\leq 1\}$,
$$
|1-g_t^{f, \eta}(u)|\leq \frac32,$$
 and assumptions $\mathbf{H}_1, \mathbf{H}_2$ hold, we have that
$$
\Delta_\eta^K:=\sup_{x\in K, t\in [0, T]}|\wt b_t^{f, \eta}(x)-\wt b(x)|\to 0, \quad \eta\to 0
$$
for any compact subset $K\subset\R^d$. Denote $\phi_t=\phi^{x,f}$ and take
$$
K=\mathrm{closure}\Big(\{\phi_s,  s\in [0, T]\} \Big), \quad K'=\{y: \mathrm{dist}(y, K)\leq 1\}.
$$
 By the assumptions $\mathbf{H}_1, \mathbf{H}_2$, there exists a constant $C_{K'}$ such that
$$
|c(y,u)|\leq C_{K'}|u|, \quad |u|\leq 1, \quad y\in K'.
$$
Denote
$$
\tau_{K'}=\inf\{t: Y_t^{x',\eta}\not \in K'\}
$$
with the usual convention $\inf\varnothing=T$. We have
$$\ba
Y_t^{x',\eta}-\phi_t&=x'-x+\int_0^t\Big(\wt b_t^{f, \eta}(\phi_s)-\wt b(\phi_s)\Big)\, ds
\\&+\int_0^t\Big(\wt b_t^{f, \eta}(Y^{x',\eta}_s)-\wt b_t^{f, \eta}(\phi_s)\Big)\, ds+\int_0^t\int_{|u|<\eta}c(Y_{s-}^{x',\eta},u) \wt Q^{f, \eta}(\di u,\di s).
\ea
$$
By the Doob maximal inequality and the It\^o isometry, for any $\eps>0$
$$\ba
\P\left(\sup_{t\in[0, \tau_{K'}]}\left|\int_0^t\int_{|u|<\eta}c(Y_{s-}^\eta,u) \wt Q^{f, \eta}(\di u,\di s)\right|\geq\eps\right)&\\&\hspace*{-3cm}\leq\frac1{\eps^2}
\E\left(\int_0^{\tau_{K'}}\int_{|u|<\eta}c(Y_{s-}^\eta,u) \wt Q^{f, \eta}(\di u,\di s)\right)^2
\\&\hspace*{-3cm}=\frac1{\eps^2}
\E\int_0^{\tau_{K'}}\int_{|u|<\eta}c(Y_{s-}^\eta,u)^2 (1+g_s^{f, \eta}(u))\mu(\di u)\di s
\\&\hspace*{-3cm}\leq C_{K'}\frac{3T}2\int_{|u|<\eta}|u|^2\mu(\di u)\to 0, \qquad \eta\to 0;
\ea$$
in the last inequality we have used that $1+g_s^{f, \eta}(u)\leq \frac32$. It is easy to check that the functions $\wt b_t^{f, \eta}$ are uniformly Lipschitz, i.e. there exists $L$ such that, for any $t\in [0, T], \eta\in (0,1]$
$$
|\wt b_t^{f, \eta}(x)-\wt b_t^{f, \eta}(y)|\leq L|x-y|.
$$
Then  on the set
$$
A_{\eps, \eta}:=\left\{\sup_{t\in[0, \tau_{K'}]}\left|\int_0^t\int_{|u|<\eta}c(Y_{s-}^\eta,u) \wt Q(\di u,\di s)\right|<\eps\right\}
$$
we have
$$
|Y_t^{x',\eta}-\phi_t|\leq |x'-x|+T\Delta_\eta^K+\eps+L\int_0^t|Y^{x',\eta}_s (\phi_s)-\phi_s|\, ds, \quad t\in [0, \tau_{K'}],
$$
which by the Gronwall inequality yields
\be\label{Gronwall}
\sup_{t\in [0, \tau_{K'}]}|Y_t^{x',\eta}-\phi_t| \leq \Big(|x'-x|+T\Delta_\eta^K+\eps\Big)e^{LT}.
\ee
Take
$$
\rho=\frac12e^{LT},
$$
and $\eps, \eta$ small enough for
$$
 \Big(T\Delta_\eta^K+\eps\Big)e^{LT}<1.
 $$
 Then, for any $\gamma\in (0,1]$ and $x'$ with $|x'-x|<\rho \gamma$ we have on the set $A_{\eps, \eta}$
 $$
 \sup_{t\in [0, \tau_{K'}]}|Y_t^{x',\eta}-\phi_t|<1.
 $$
 Because trajectories of $Y^{x',\eta}$ are right continuous this yields that, on this set, $\tau_{K'}=T$ and actually
 \be\label{Gronwall_2}
\sup_{t\in [0,T]}|Y_t^{x',\eta}-\phi_t| \leq \Big(|x'-x|+T\Delta_\eta^K+\eps\Big)e^{LT}.
\ee
Now, for a given $\gamma$, we take $\eps<\frac14\gamma e^{-LT}$ and    we get that for $\eta$ small enough
$$
\inf_{x'\in B(x, \rho\gamma)}\P(\sup_{t\in [0,T]}|Y_t^{x',\eta}-\phi_t|\leq \gamma)\geq \P(A_{\eps, \eta}),
$$
and because
$$
\P(A_{\eps, \eta})\to 1, \quad \eta\to 0
$$
this completes the proof. \end{proof}

Now we are ready to complete the proof of the Key Lemma. Denote
$$\ba
\mathcal{E}^{f, \eta}=\exp\left(-\int_{\R^d\times [0, T]}\log (1+g_t^{f, \eta}(u))\wt Q^{f, \eta}(\di u, \di t)\right.&
\\&\hspace*{-3.5cm}\left.+\int_{\R^d\times [0, T]}\Big(g_t^{f, \eta}(u)-\log (1+g_t^{f, \eta}(u))\Big)\, \mu(\di u)\di t\right),
\ea$$
recall that by the construction $1+g_t^{f, \eta}(u)\in [\frac12, \frac32]$ and $1+g_t^{f, \eta}(u)=1$  whenever either $|u|\geq \eta$ or $|u|\leq \zeta^{f, \eta}$. Hence the stochastic integral under the exponent is well defined in the standard (quadratic) It\^o sense. The second (deterministic) integral is even bounded; hence it is easy to see that
\be\label{E_bound}
P\left(\mathcal{E}^{f, \eta}<\frac{1}{N}\right)=P(\log\mathcal{E}^{f, \eta}<-\log N)\to 0, \quad N\to \infty
\ee
for any $\eta\in (0,1]$.

On the other hand, the classical result by Skorokhod \cite{Skor} tells us that the laws of the PPMs $N(\di u, \di t)$ and $Q^{f, \eta}(\di u, \di t)$ are equivalent, and $\mathcal{E}^{f, \eta}$ equals to the Radon-Nikodym derivative of $\mathrm{Law}(N)$ w.r.t. $\mathrm{Law}(Q^{f, \eta})$ evaluated at $Q^{f, \eta}$. Recall that $X^{x',S,\eta, \mathrm{trunc}}$ and $Y^{x',S,\eta}$ are defined as the strong solutions to the same SDE with the noises $N$ and $Q^{f, \eta}$, respectively. Hence we can treat them as images of $N, Q^{f, \eta}$ under some  measurable mapping, which  yields the identity
$$
\Pp\left(\sup_{t\in [S, Q]}|X_t^{x',S,\eta, \mathrm{trunc}}-\phi_t^{x,S, f}|\leq \gamma\right)=\E 1_{\sup_{t\in [S, Q]}|Y_t^{x',S,\eta, }-\phi_t^{x,S, f}|<\eps}\mathcal{E}^{f, \eta}.
$$

By Lemma \ref{l3}, for a given $x\in \R^d$, $\gamma>0$, and $f\in  \mathbf{F}^{\mathrm{step}}_{0,T}$,  there exists $\eta^{f,x, \gamma}>0$ such that
\be\label{2/3}
\inf_{x'\in B(x, \gamma\rho), 0\leq S\leq Q\leq T}\Pp\left(\sup_{t\in [S, Q]}|Y_t^{x',S,\eta}-\phi_t^{x,S, f}|\leq \gamma\right)\geq\frac23
\ee
for any $\eta\in (0, \eta^{f,x, \gamma}]$. By \eqref{E_bound}, for any such  $\eta$  there exists $N^{f, \eta}$ large enough for
$$
P(\mathcal{E}^{f, \eta}\geq \frac{1}{N^{f, \eta}})\geq \frac23.
$$
Then for any $x'\in B(x, \gamma\rho), 0\leq S\leq Q\leq T$ we have
$$
\Pp\left(\sup_{t\in [S, Q]}|Y_t^{x',S,\eta}-\phi_t^{x,S, f}|\leq \gamma, \mathcal{E}^{f, \eta}\geq \frac{1}{N^{f, \eta}} \right)\geq \frac13
$$
and
$$\ba
\Pp\left(\sup_{t\in [S, Q]}|X_t^{x',S,\eta, \mathrm{trunc}}-\phi_t^{x,S, f}|\leq \gamma\right)&\\&\hspace*{-3cm}\geq  \frac{1}{N^{f, \eta}}\Pp\left(\sup_{t\in [S, Q]}|Y_t^{x',S,\eta}-\phi_t^{x,S, f}|\leq \gamma, \mathcal{E}^{f, \eta}\geq \frac{1}{N^{f, \eta}} \right)\\&\hspace*{-3cm}\geq \frac1{3N^{f, \eta}}>0,
\ea
$$
which completes the proof. \qed

\appendix
\section{Proof of \eqref{phi_delta}}\label{sA}
For given $s_1, \dots, s_K$ define  $\lambda^{s_1, \dots, s_K}\in \Lambda_{0,T}$ as a function,  linear on each of the intervals $[0, t_1], \dots [t_K, T]$,
and satisfying
$$
\lambda^{s_1, \dots, s_K}_{t_j}=s_j, \quad j=1, \dots, K.
$$
Denote
$$
\wt\phi^{s_1, \dots, s_K}_t=\phi^{s_1, \dots, s_K}_{\lambda^{s_1, \dots, s_K}_t}, \quad f^{s_1, \dots, s_K}_t=\phi^{s_1, \dots, s_K}_{\lambda^{s_1, \dots, s_K}_t},\quad t\in [0, T].
$$
It is clear that, for a fixed set $\{t_k\}$,
\be\label{der}
\sup_{{|t_k-s_k|<\delta, k=1, \dots,K}}\left|\frac{\di \lambda^{s_1, \dots, s_K}_t}{\di t}-1\right|\to 0, \quad \delta\to 0,
\ee
and, for a given $f\in \mathbf{F}^{\mathrm{step}}_{0,T}$,
\be\label{f}
\sup_{{|t_k-s_k|<\delta, k=1, \dots,K}}\|f^{s_1, \dots, s_K}-f\|_{L_1(0, T)}\to 0, \quad \delta\to 0.
\ee
By \eqref{der} we have
$$
\sup_{{|t_k-s_k|<\delta, k=1, \dots,K}} \sup_{0\leq s<t\leq T}\left|\log\frac{\lambda_t^{s_1, \dots, s_K}-\lambda_s^{s_1, \dots, s_K}}{t-s}\right|\to 0, \quad \delta\to 0,
$$
hence in order to prove \eqref{phi_delta} it is sufficient to prove that
\be\label{phi_Delta}
\sup_{{|t_k-s_k|<\delta, k=1, \dots,K}}\sup_{t\in [0, T]}|\wt\phi^{s_1, \dots, s_K}_t-\phi_t|\to 0, \quad \delta\to 0.
\ee
The functions $\wt\phi^{s_1, \dots, s_K},\phi$  have the same set $\{t_k\}$ of discontinuity points. Moreover,
$$
\wt\phi^{s_1, \dots, s_K}_{t_k}=\phi_{t_k}= y_k, \quad k=0, \dots, K,
$$
and on each of the time intervals $[0, t_1), [t_1, t_2), \dots, [t_K, T]$ function $\wt\phi^{s_1, \dots, s_K}$ satisfies the ODE
$$
\di \wt \phi_t^{s_1, \dots, s_K}=\Big(\widetilde b(\wt \phi_t^{s_1, \dots, s_K})+\sigma(\wt \phi_t^{s_1, \dots, s_K}) f^{s_1, \dots, s_K}_t\Big)\frac{\di \lambda^{s_1, \dots, s_K}_t}{\di t}\, \di t.
$$
Function $\phi$ at each of these intervals satisfies a similar ODE \eqref{ODE} with the same initial value; in addition, it is easy to verify that the values of the functions $ f^{s_1, \dots, s_K}$, $\wt \phi_t^{s_1, \dots, s_K}$ are uniformly bounded for $|t_k-s_k|<\delta, k=1, \dots,K$. This yields, by the Lipschitz continuity of $\wt b, \sigma$,
$$\ba
|\phi_t^{s_1, \dots, s_K}-\phi_t|&\leq C\sup_{s\in [0, T]}\left|\frac{\di \lambda^{s_1, \dots, s_K}_s}{\di s}-1\right|+C\|f^{s_1, \dots, s_K}-f\|_{L_1(0, T)}
\\&+C\int_{t_k}^t|\phi^{s_1, \dots, s_K}_s-\phi_s|\, \di s, \quad t\in [t_k,t_{k+1}), \quad k=0, \dots, K
\ea
$$
with a constant $C$ which does not depend on the choice of $\{s_k\}$.  By the Gronwall inequality, this gives
$$
\sup_{t\in [0, T]}|\phi_t^{s_1, \dots, s_K}-\phi_t|\leq Ce^{CT}\left(\sup_{s\in [0, T]}\left|\frac{\di \lambda^{s_1, \dots, s_K}_s}{\di s}-1\right|+\|f^{s_1, \dots, s_K}-f\|_{L_1(0, T)}\right).
$$
Combined with \eqref{der},\eqref{f} this proves \eqref{phi_Delta} and completes the proof of \eqref{phi_delta}.\qed

\end{document}